\documentclass[12pt]{article}

\usepackage{amsmath,amssymb,amsthm,url,tikz,enumitem,multirow,geometry}
\usepackage{subcaption,longtable,rotating,color,verbatim}

\usetikzlibrary{intersections,arrows,decorations.markings}

\tikzset{->-/.style={decoration={
  markings,
  mark=at position #1 with {\arrow{>}}},postaction={decorate}}}  
\tikzset{-<-/.style={decoration={
  markings,
  mark=at position #1 with {\arrow{<}}},postaction={decorate}}}

\newcommand{\cA}{\mathcal{A}}

\newcommand{\cC}{\mathcal{C}}
\newcommand{\cD}{\mathcal{D}}

\newcommand{\cG}{\mathcal{G}}

\newcommand{\cS}{\mathcal{S}}

\newcommand{\cL}{\mathcal{L}}
\newcommand{\bZ}{\mathbb{Z}}

\newtheorem{theorem}{Theorem}[section]

\newtheorem{cor}[theorem]{Corollary}
\newtheorem{proposition}[theorem]{Proposition}
\newtheorem{lemma}[theorem]{Lemma}
\newtheorem{question}{Question}

\begin{document}
\title{On which groups can arise as the canonical group of a spherical latin bitrade}
\author{Kyle Bonetta-Martin\thanks{Research supported by a London Mathematical Society Undergraduate Research Bursary, grant number URB 15--49.} \ and Thomas A. McCourt}
\date{\small Department of Mathematics and Statistics, Plymouth University, Drake Circus, Plymouth PL4 8AA.
}

\maketitle

\setcounter{footnote}{1}

\begin{abstract}
We address a question of Cavenagh and Wanless asking: which finite abelian groups arise as the canonical group of a spherical latin bitrade? We prove the existence of an infinite family of finite abelian groups that do not arise as canonical groups of spherical latin bitrades. 
Using a connection between abelian sandpile groups of digraphs underlying directed Eulerian spherical embeddings, we go on to provide several, general, families of finite abelian groups that do arise as canonical groups. These families include:
\begin{itemize}
\item any abelian group in which each component of the Smith Normal Form has composite order;
\item any abelian group with Smith Normal Form $\mathbb{Z}^{n}_p\oplus\left(\bigoplus_{i=1}^k\mathbb{Z}_{pa_i}\right)$, where $1\leq k$, $2\leq a_1,a_2,\ldots, a_k,p$ and $n\leq 1+2\sum_{i=1}^k(a_i - 1)$; and
\item with one exception and three potential exceptions any abelian group of rank two.
\end{itemize} 
\end{abstract}

\section{Introduction}

Given two latin squares of the same order a latin trade describes the differences between them. Early motivation \cite{DraKep1} for their study arose from considering the differences between the operation tables of a finite group and a latin square of the same order, that is: what is the `distance' between a group and a latin square? 
The study of the topological and geometric properties of latin trades has lead to significant progress towards understanding such differences, see for example \cite{CavWan, DraHamKal, BlackburnTMcC, TMcC, Szabados}, also see \cite{Cav-surv} for a survey of earlier results. 

Given a latin trade it may  be the case that the constituent partial latin squares are not `contained' (do not embed) in any group operation table, \cite{CavWan}. Hence, it is desirable to identify those that are. Connected latin bitrades of maximum size, equivalently spherical latin bitrades provide a family of latin bitrades for which the constituent partial latin squares do embed. We are interested in the `minimal group' that such constituent partial latin squares embed in, and indeed what groups arise as such minimal groups.

\subsection{Spherical latin bitrades}

A \textit{partial latin square} $P$ is an $\ell\times m$ array, in which the cells either contain an element of a set $S$ of symbols or are empty, such that each row and each column contains each of the symbols of $S$ at most once. Without loss of generality we let $S=\{s_1,s_2,\ldots,s_n\}$ and index the rows and columns by the sets $R=\{r_1,r_2,\ldots,r_\ell\}$ and $C=\{c_1,c_2,\ldots,c_m\}$ respectively (we may assume that each symbol in $S$ occurs at least once in the array and the rows of $R$ and columns of $C$ are all nonempty). As such a partial latin square $P$ can be considered to be a subset of $R\times C\times S$ such that if $(r_1,c_1,s_1)$ and $(r_2,c_2,s_2)$ are distinct triples in $P$, then at most one of $r_1=r_2$, $c_1=c_2$ and $s_1=s_2$ holds. 

A \textit{latin bitrade} is an ordered pair, $(W,B)$ say, of non-empty partial latin squares such that for each triple $(r_i,c_j,s_k)\in W$ (respectively $B$) there exists unique $r_{i'}\neq r_i$, $c_{j'}\neq c_j$ and $s_{k'}\neq s_k$ such that 
$$\big\{(r_{i,'}c_j,s_k),(r_i,c_{j'},s_k),(r_i,c_j,s_{k'})\big\}\subset B\text{ (respectively $W$)}.$$
Note that $(W,B)$ is a latin bitrade if and only if $(B,W)$ is also a latin bitrade. The \textit{size} of such a latin bitrade is $|W|$ (equivalently $|B|$). A latin bitrade $(W,B)$ for which there does not exist any latin bitrade $(W',B')$ such that $W'\subsetneq W$ and $B'\subsetneq B$ is said to be \textit{connected}.

Let $(W,B)$ be a latin bitrade; for each row, $r$ say, of $(W,B)$ a permutation $\rho_r$ of the symbols in row $r$ can be defined by $\rho_r(s)=s'$ if and only if $(r,c,s)\in W$ and $(r,c,s')\in B$ for some $c$ in $C$. A row $r$ for which $\rho_r$ is comprised of a single cycle is said to be \textit{separated}. Similar definitions hold for separated columns and separated symbols. A latin bitrade in which each row, each column and each symbol is separated is called a \textit{separated latin bitrade}. 
Suppose that $(W,B)$ is a latin bitrade which is not separated. Then replacing each non-separated row $x$ (respectively column, symbol) by new rows (respectively columns, symbol) for each of the cycles in $\rho_x$ we obtain a separated latin bitrade. See the survey paper \cite{Cav-surv} for further details and discussion.

A connected latin bitrade $(W,B)$ can be used to construct a face two-coloured triangulation $\cG_{W,B}$ of a pseudo-surface $\Sigma$ in which the vertex set is $R\sqcup C\sqcup S$ and there is an edge between a pair of vertices if and only if the vertices occur together in a triple of $W$ (equivalently a triple of $B$). For each triple $(r,c,s)\in W$ a white triangular face with vertices $r,c,s$ is constructed and for each $(r',c',s')\in B$ a black triangular face with vertices $r',c',s'$ is constructed. As $(W,B)$ is a bitrade the graph underlying $\cG_{W,B}$ is simple, and as $(W,B)$ is connected $\cG_{W,B}$ is also connected. The pseudo-surface $\Sigma$ is a true surface if the rotation at each vertex is a full rotation; this occurs if and only if $(W,B)$ is separated (in which case each row, column or symbol permutation corresponds to the rotation at the corresponding vertex). If $\Sigma$ is not a surface, then replacing each pinch point of multiplicity $t$ with a $t$ vertices, one on each of the sheets at the pinch point, corresponds to the above construction taking a non-separated bitrade to a separated one. As the triangulation $\cG_{W,B}$ is face two-coloured and the underlying graph is vertex three-coloured it follows, see the proof of Theorem 10.1 in \cite{GrannellGriggs}, that $\cG_{W,B}$ is orientable.

The \textit{genus} of a separated connected latin bitrade is the genus of the surface obtained in the above manner; in particular separated connected latin bitrades of genus zero are referred to as \textit{spherical latin bitrades}. Note that for any connected latin bitrade of size $\ell$ we have that $|R|+|C|+|S|\leq \ell+2$, with equality if and only if the bitrade is a spherical latin bitrade, see \cite{BlackburnTMcC}. That is, spherical latin bitrades are the connected latin bitrades of minimal size (with respect to the sum of the number of rows, columns and symbols).

In \cite{CavLis} Cavenagh and Lison\v{e}k prove the following result.

\begin{theorem}[Cavenagh \& Lison\v{e}k, \cite{CavLis}]
\label{thm:CavLis}
Spherical latin bitrades are equivalent to spherical Eulerian triangulations whose underlying graphs are simple.
\end{theorem}

Note that an Eulerian graph that has an embedding in the sphere is necessarily vertex three-colourable \cite{Hea}. It is not hard to generalise Theorem \ref{thm:CavLis} to surfaces of higher genus, however as face two-coloured triangulations of surfaces of higher genus may not be vertex three-colourable, an additional condition is required.

\begin{cor}
Separated connected latin bitrades of genus $g$ are equivalent to vertex three-colourable Eulerian triangulations of genus $g$ whose underlying graphs are simple. 
\end{cor}

\subsection{Embeddings of latin bitrades into abelian groups}

Two partial latin squares are said to be \textit{isotopic} if they are equal up to a relabelling of their sets of rows, columns and symbols. A partial latin square $P$, with row set $R$, column set $C$ and symbol set $S$, is said to \textit{embed in an 
abelian group $\Gamma$} if there exist injective maps $\phi_1:R\rightarrow \Gamma$, $\phi_2:C\rightarrow \Gamma$ and $\phi_3:S\rightarrow \Gamma$ such that $\phi_1(r)+\phi_2(c)=\phi_3(s)$ for all $(r,c,s)\in P$. In other words $P$ is isotopic to a partial latin square contained in the operation table of $\Gamma$. See Figure \ref{fig:Kyle-example} for an example. 

By defining $\phi|_R=\phi_1$, $\phi|_C=\phi_2$, and $\phi|_S=-\phi_3$ it follows, see \cite{BlackburnTMcC}, that $P$ embeds in an abelian group $\Gamma$ if and only if there exists a function $\phi:R\sqcup C\sqcup S\rightarrow \Gamma$ that is injective when restricted to each of $R$, $C$ and $S$ and is such that $\phi(r)+\phi(c)+\phi(s)=0$ for all $(r,c,s)\in P$. The map $\phi$ is called an \textit{embedding} of $P$.
An abelian group $\Gamma$ is said to be a \textit{minimal abelian representation} of a partial latin square $P$ if $P$ embeds in $\Gamma$ and the image of $\phi$ generates $\Gamma$ for all embeddings $\phi$ of $P$ in $\Gamma$. 

\begin{figure}[!h]
\begin{center}
\begin{tabular}{ccc}
\begin{tabular}{|ccc|}
\hline
$a$ & $b$ & $c$ \\
$c$ &  & $a$ \\
& $a$ & $b$ \\
\hline
\end{tabular}
&$\qquad$ &
\begin{tabular}{c|cccc|}
 $+$ & $0$ & $1$ & $2$ & $3$ \\
\hline
$0$ & $0$ & $\mathbf{1}$ & $\mathbf{2}$ & $\mathbf{3}$ \\
$1$ & $1$ & $2$ & $3$ & $0$ \\
$2$ & $2$ & $\mathbf{3}$ & $0$ & $\mathbf{1}$ \\
$3$ & $3$ & $0$ & $\mathbf{1}$ & $\mathbf{2}$ \\
\hline
\end{tabular}
\end{tabular}
\end{center}
\caption{The partial latin square above left embeds into $\bZ_4$ as illustrated by the bold faced entries in the operation table of $\bZ_4$, above right.}
\label{fig:Kyle-example}
\end{figure}

Two partial latin squares are said to be \textit{conjugate} if they are equal up to permutations of the roles of rows, columns and symbols. Two partial latin squares, say $P$ and $Q$, for which a partial latin square isotopic to $P$ is conjugate to a partial latin square isotopic to $Q$ are said to be 
in the same \textit{main class}. Note that if a partial latin square $P$ has an embedding in an abelian group $\Gamma$, every partial latin square in the same main class as $P$ also has an embedding in $\Gamma$. 

As we are interested in embeddings (into abelian groups) of partial latin squares (and given that if a partial latin square $P$ embeds in an abelian group $\Gamma$, so does any partial latin square isotopic to $P$) from here on we will assume that the row, column and symbol sets of a partial latin square are pairwise disjoint. 

In \cite{CavDra} Cavenagh and Dr\'apal asked the following questions ``Can the individual partial latin squares of a connected separated latin bitrade be embedded into the operation table of an abelian group? If this is not true in general is it true for spherical latin bitrades?''. The case of spherical latin bitrades was solved by Cavenagh and Wanless in \cite{CavWan} and independently by Dr\'apal, H\"am\"al\"ainen and Kala in \cite{DraHamKal}. Cavenagh and Wanless \cite{CavWan} also showed that separated connected latin bitrades of higher genus exist for which the constituent partial latin squares do not embed in any group. Hence our focus on spherical latin bitrades.

Let $P$ be a partial latin square with row set $R$, column set $C$ and symbol set $S$. Let $V=R\cup C\cup S$ and define an abelian group $\cA_P$ with generating set $V$ subject to the relations $\{r+c+s=0:(r,c,s)\in P\}$. 
Note that, if $P$ and $Q$ are two partial latin squares in the same main class, then $\cA_P\cong \cA_Q$. Also, note that two partial latin squares, $P$ and $Q$, from different main classes may also satisfy $\cA_P\cong \cA_Q$ (see Figure 2  in \cite{TMcC}).

The group $\cA_P$ has the `universal' property that any minimal abelian representation of $P$ is a quotient of $\cA_P$, \cite{DraKep}, also see \cite{BlackburnTMcC}.  Moreover $\cA_P$ is of the form $\bZ\oplus\bZ\oplus\cC_P$, again see \cite{BlackburnTMcC}. Dr\'apal et al \cite{DraHamKal} and Cavenagh and Wanless \cite{CavWan} proved that $\cC_W$ is finite when $(W,B)$ is a spherical latin bitrade. So in this case $\cC_W$ is the torsion subgroup of $\cA_W$.
Cavenagh and Wanless conjectured that $\cC_W\cong \cC_B$ (and hence $\cA_W\cong \cA_B$), \cite{CavWan}, also see \cite{Kou,BCC}. This is indeed the case. 
\begin{theorem}[Blackburn \& McCourt \cite{BlackburnTMcC}]
\label{thm:BlackburnTMcC}
Let $(W,B)$ be a spherical latin bitrade, then $\cA_W\cong\cA_B\cong\bZ\oplus\bZ\oplus\cC$, where $\cC$ is finite.
\end{theorem}
The group $\cC$ in Theorem \ref{thm:BlackburnTMcC} is referred to as the \textit{canonical group} of the spherical latin bitrade (see \cite{GruWan, TMcC}).

In \cite{CavWan} Cavenagh and Wanless asked the following question.
\begin{question}
\label{ques:main}
Which abelian groups arise as the canonical group of a spherical latin bitrade?\footnote{Cavenagh and Wanless actually asked this for the finite torsion subgroup of $\cA_W$ as Theorem \ref{thm:BlackburnTMcC} was not established at the time.}
\end{question}
It is this question that we address in this paper.
For any cyclic group $\bZ_n$ the existence of spherical latin bitrades whose canonical group is isomorphic to $\bZ_n$ was established by Cavenagh and Wanless in \cite{CavWan}. They also noted that no spherical latin bitrade exists whose canonical group is isomorphic to $\bZ_2\oplus\bZ_2$.

Given a face 2-coloured triangulation of the sphere in which the underlying graph is not necessarily simple and leaving the definitions of $\cA_W$ and $\cA_B$ unchanged it is still the case that $\cA_W\cong\cA_B\cong\bZ\oplus\bZ\oplus\cC$ where $\cC$ is finite \cite{BlackburnTMcC}. In \cite{TMcC} the second author showed that given any finite abelian group $\Gamma$ there exists a face 2-coloured triangulation of the sphere whose canonical group is isomorphic to $\Gamma$. However, unless $\Gamma$ is a cyclic group the triangulations constructed have underlying graphs that are not simple.

In Section \ref{sec:exist} we prove the existence of several, general, infinite families of abelian groups  that arise as canonical groups of spherical latin bitrades. Before doing so, we first prove that there exist infinitely many abelian groups that do not arise as the canonical group of any spherical latin bitrade. 

\begin{theorem}
\label{thm:non-existence}
There does not exist a spherical latin bitrade whose canonical group is isomorphic to $\bZ_2^k$ for any $k\geq 2$.
\end{theorem}

\begin{proof}
In the following we will make repeated use of the fact that for $u,v,w,x,y,z\in \bZ_2^k$, if $u+w=y$, $v+w=z$ and $v+x=y$, then $u+x=z$.

Let $k\geq 2$ and suppose that $(W,B)$ is a spherical latin bitrade whose canonical group is isomorphic to $\bZ_2^k$. So, by Theorem \ref{thm:BlackburnTMcC}, both $W$ and $B$ embed in $\bZ^k_2$. 

Recall that we may assume that the row, column and symbol sets of $W$ (and of $B$) are pairwise disjoint; denote them, respectively, by $R=\{r_1,r_2,\ldots, r_\ell\}$, $C=\{c_1,c_2,\ldots, c_m\}$ and $S=\{s_1,s_2,\ldots, s_n\}$. Let $\cG_{W,B}$ be the related triangulation and $G$ be the underlying graph of this triangulation. As $\cG_{W,B}$ has a proper face 2-colouring, $G$ is Eulerian, and, as $(W,B)$ is a latin bitrade, the minimum degree of $G$ is at least four. Moreover, $\cG_{W,B}$ is a triangulation of the sphere, so, by Euler's formula, $G$ contains at least six vertices of degree four. 

As spherical latin bitrades in the same main class all have isomporphic canonical groups, without loss of generality, we may assume that the degree of $r_1$ is four, and $(r_1,c_1,s_1),(r_1,c_2,s_2)\in B$ and $(r_1,c_1,s_2), (r_1,c_2,s_1)\in W$ where $c_1\neq c_2$ and $s_1\neq s_2$.
Hence, as $(W,B)$ is a latin bitrade, there exist $x_1,x_2,x_3,x_4\in R\setminus\{r_1\}$ such that $(x_1,c_2,s_1),(x_3,c_1,s_2)\in B$ and $(x_2,c_1,s_1),(x_4,c_2,s_2)\in W$ (see Figure \ref{fig:CaseC} for an illustration of the corresponding faces).

\begin{figure}[!hb]
\begin{center}
\scalebox{0.9}
{\begin{tikzpicture}[fill=gray!50, scale=1, vertex/.style={circle,inner sep=2,fill=black,draw}, dot/.style={circle,inner sep=0.7,fill=black,draw}]

\coordinate (r1) at (2,2);
\coordinate (c1) at (3,3);
\coordinate (s1) at (1,3);
\coordinate (c2) at (1,1);
\coordinate (s2) at (3,1);
\coordinate (x1) at (4,2);
\coordinate (x2) at (2,4);
\coordinate (x3) at (0,2);
\coordinate (x4) at (2,0);


\filldraw[color=gray!50] (r1) -- (c1) -- (s1) -- cycle;
\filldraw[color=gray!50] (r1) -- (c2) -- (s2) -- cycle;
\filldraw[color=gray!50] (x1) -- (c1) -- (s2) -- cycle;
\filldraw[color=gray!50] (x3) -- (c2) -- (s1) -- cycle;


\draw[thick] (x1) -- (x2) -- (x3) -- (x4) -- cycle;
\draw[thick] (c1) -- (s1) -- (c2) -- (s2) -- cycle;
\draw[thick] (s1) -- (s2);
\draw[thick] (c1) -- (c2);


\node at (3.6,3.85) [dot]{};
\node at (3.75,3.75) [dot]{};
\node at (3.85,3.6) [dot]{};

\node at (3.6,0.15) [dot]{};
\node at (3.75,0.25) [dot]{};
\node at (3.85,0.4) [dot]{};

\node at (0.15,3.6) [dot]{};
\node at (0.25,3.75) [dot]{};
\node at (0.4,3.85) [dot]{};

\node at (0.15,0.4) [dot]{};
\node at (0.25,0.25) [dot]{};
\node at (0.4,0.15) [dot]{};


\node at (r1) [vertex,label=east:$r_1$]{};
\node at (c1) [vertex,label=north east:$c_1$]{};
\node at (s1) [vertex,label=north west:$s_1$]{};
\node at (c2) [vertex,label=south west:$c_2$]{};
\node at (s2) [vertex,label=south east:$s_2$]{};
\node at (x1) [vertex,label=east:$x_1$]{};
\node at (x2) [vertex,label=north:$x_2$]{};
\node at (x3) [vertex,label=west:$x_3$]{};
\node at (x4) [vertex,label=south:$x_4$]{};

\end{tikzpicture}}
\end{center}
\caption{The vertex $r_1$ and nearby faces in $\cG_{W,B}$.}
\label{fig:CaseC}
\end{figure}
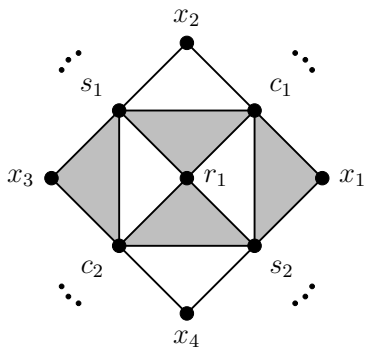

As $W$ embeds in $\bZ_2^k$, $x_2=x_4$ and, as $B$ embeds in $\bZ_2^k$, $x_1=x_3$. Suppose that $x_1=x_2$. Let 
$$W'=\{(r_1,c_1,s_2),(r_1,c_2,s_1),(x_1,c_1,s_1),(x_1,c_2,s_2)\}$$
and
$$B'=\{(r_1,c_1,s_1),(r_1,c_2,s_2),(x_1,c_1,s_2),(x_1,c_2,s_1)\}.$$
Then $(W',B')$ is a spherical latin bitrade such that $W'\subseteq W$ and $B'\subseteq B$. As $(W,B)$ is connected, it must be the case that $W'=W$ and $B'=B$. However, the canonical group of $(W',B')$ is $\bZ_2$, a contradiction. So $x_1\neq x_2$; in which case $G$ contains a subgraph $H=(V,E)$ where $V=\{r_1,x_1,x_2,c_1,s_1,s_2\}$ and $E=\{r_1c_1,r_1s_1,r_1s_2, x_1c_1,x_1s_1,x_1s_2, x_2c_1,x_2s_1,x_2s_2\}$. However, $H$ is isomorphic to $K_{3,3}$; which contradicts $\cG_{W,B}$ being a spherical embedding.
\end{proof}

\section{Existence results}
\label{sec:exist}

\subsection{Directed Eulerian spherical embeddings}

Let $D$ be a, not necessarily simple, digraph of order $n$ with vertex set $V(D)=\{v_1,v_2,\ldots v_n\}$. The \textit{adjacency matrix} $A=[a_{ij}]$ of $D$ is the $n\times n$ matrix where entry $a_{ij}$ is the number of arcs from vertex $v_i$ to vertex $v_j$. The \textit{asymmetric Laplacian} of $D$ is the $n\times n$ matrix $L(D)=B-A$ where $B$ is the diagonal matrix in which entry $b_{ii}$ is the out-degree of vertex $v_{i}$.
The digraph $D$ is said to be \textit{Eulerian} if, for each $v\in V(D)$, the out-degree at $v$ equals the in-degree at $v$. Hence, in an Eulerian digraph we will simply refer to the degree of a vertex $v$, i.e. $\deg v$.

Let $D$ be a connected Eulerian digraph of order $n$ with vertex set $V(D)=\{v_1,v_2,\ldots v_n\}$. Fix an $i$, where $1\leq i\leq n$ and define $L'(D,i)$ to be the matrix obtained by removing row and column $i$ from $L(D)$. As $D$ is connected and Eulerian, the group $\bZ^{n-1}/L'(D,i)\bZ^{n-1}$ is invariant of the choice of $i$, see \cite[Lemma 4.12]{HolLevMesPerProWil}. Hence, the \textit{abelian sandpile group} of the connected Eulerian digraph $D$ can be defined to be the group $\cS(D)=\mathbb{Z}^{n-1}/\mathbb{Z}^{n-1}L'(D,n)$; moreover $\cS(D)\cong\mathbb{Z}^{n-1}/\mathbb{Z}^{n-1}L'(D,i)$, for any $1\leq i\leq n$.

Consider an embedding $\cD$ of a connected Eulerian digraph $D$ in an orientable surface $S$. If each face of the embedding corresponds to a directed cycle in $D$, equivalently the rotation at each vertex alternates between incoming and outgoing arcs, then the embedding is said to be a \textit{directed Eulerian embedding}, see \cite{BonConMorMcK,BonHarSir}. If the embedding is in the sphere we call it a \textit{directed Eulerian spherical embedding}. 

Suppose that $\cG$ is a face two-coloured triangulation of the sphere. By \cite{Hea}, the underlying digraph of $\cG$ has a vertex three-colouring with colour classes $R$, $C$ and $S$. Tutte \cite{Tutte} described a construction, from $\cG$, of directed Eulerian spherical embeddings $D_I(\cG)=D_I$ with vertex set $I$, where $I\in\{R,C,S\}$. We give a description of the construction from $\cite{TMcC}$.

Let $\{I,I_1,I_2\}=\{R,C,S\}$. Consider a vertex $v_i\in I$. Then $v_i$ has even degree, say $d$, the rotation at $i$ is 
$(u_1,v_1,u_2,v_2,\ldots, u_{d/2},v_{d/2}),$
where, without loss of generality, $u_j\in I_1$ and $v_j\in I_2$ for all $1\leq j\leq d/2$ and the edge $e_j$ between $u_j$ and $v_j$ in the rotation is contained in a black face. 
Then in $D_I$ there are $d/2$ outgoing arcs from vertex $v_i$, say $a_j$, $1\leq j\leq d/2$, one for each black face, and the terminal vertex for arc $a_j$ is the vertex in $I$ contained in the white face containing edge $e_j$. Clearly, $D_I$ inherits a spherical embedding from $\cG$ in which the arc rotation at each vertex alternates between incoming and outgoing arcs, so $D_I$ has a directed Eulerian spherical embedding. As the sphere is connected the graph underlying $D_I$ is connected. 
Note that given any of $D_R$, $D_C$ or $D_S$ the original face two-coloured triangulation can be obtained by reversing the above construction:

\begin{lemma}[Tutte, \cite{Tutte}]
\label{lem:gobackwards}
Given a directed Eulerian spherical embedding $D$, there exists a face $2$-coloured spherical triangulation $\cG$ with a vertex $3$-colouring given by the vertex sets $R$, $C$ and $S$, such that for some $I\in\{R,C,S\}$,
$D_I(\cG)\cong D.$
\end{lemma}

Tutte's Trinity Theorem \cite{Tutte} states that $|\cS(D_R)|=|\cS(D_C)|=|\cS(D_S)|$. For a spherical latin bitrade $(W,B)$ with corresponding face two-coloured triangulation $\cG$, this result was strengthened implicitly in  \cite{BlackburnTMcC} and explicitly in \cite{TMcC} to $\cS(D_R)\cong\cS(D_C)\cong\cS(D_S)\cong \cA_W\cong\cA_B$.

Given an arbitrary directed Eulerian spherical embeddings applying the above construction in reverse yields a face two-coloured triangulation. However, the underlying graph is not necessarily simple. In order to make use of the above equivalences (between sandpile groups and canonical groups of spherical latin squares) we make use of the following result.

\begin{proposition}[McCourt, \cite{TMcC}]
\label{prop:simple}
Suppose that $\cD$ is a directed Eulerian spherical embedding with underlying digraph $D$. Further suppose that $D$ is connected, has no loops, no cut vertices and its underling graph has no 2-edge-cuts. Then there exists a spherical latin bitrade whose canonical group is isomorphic to $\cS(D)$.
\end{proposition}

Hence, in order to construct a spherical latin bitrade with canonical group $\Gamma$ it suffices to find a directed Eulerian spherical embedding satisfying the connectivity conditions of Proposition \ref{prop:simple} whose abelian sandpile group is isomorphic to $\Gamma$.

\subsection{Arbitrary rank}

In this section we will construct families of canonical groups that have arbitrary rank. We will make repeated use of the following, elementary lemma. 

\begin{lemma} 
\label{cl:prime+comps}
Let $2\leq p,a$ and $0\leq x,y,\ell$. Further let $r=p(x+1)+a-x-1$, $s=p(y+1)+a-y-1$ and $t_{i,j}\in\bZ$, for $1\leq i\leq m$ and $1 \leq j \leq \ell$. Then the matrix
\begin{center}
\resizebox{\linewidth}{!}{$L=\left[\begin{tabular}{cc|ccc|ccc|c|ccc}
$p$ & $-p+1$ & $0$ & $\cdots$ & $0$ & $0$ & $\cdots$ & $0$ & $-1$ & $0$ & $\cdots$ &$0$\\
$-p$ & $r$ &$-p$& $\cdots$ & $-p$ & $0$ & $\cdots$ & $0$ & $x+1-a$ & $0$ & $\cdots$ &$0$\\
\hline
$0$ & $-p+1$ &  \multicolumn{3}{c|}{\multirow{3}{*}{$p\mathbb{I}_x$}}&$0$ & $\cdots$ &$0$&$-1$& $0$ & $\cdots$ &$0$\\
$\vdots$ & $\vdots$ &    \multicolumn{3}{c|}{}&$\vdots$ &$\ddots$ &$\vdots$&$\vdots$& $\vdots$ &$\ddots$ &$\vdots$\\
$0$ & $-p+1$ &    \multicolumn{3}{c|}{}&$0$ &$\cdots$ &$0$&$-1$& $0$ &$\cdots$ &$0$\\
\hline
$0$ & $-1$ &  $0$ &$\cdots$ &$0$&\multicolumn{3}{c|}{\multirow{3}{*}{$p\mathbb{I}_y$}}&$-p+1$&$0$ &$\cdots$ &$0$\\
$\vdots$ & $\vdots$ &    $\vdots$ &$\ddots$ &$\vdots$&\multicolumn{3}{c|}{}&$\vdots$&$\vdots$ &$\ddots$ &$\vdots$\\
$0$ & $-1$ &    $0$ &$\cdots$ &$0$&\multicolumn{3}{c|}{}&$-p+1$&$0$ &$\cdots$ &$0$\\
\hline
$0$ & $y+1-a$ & $0$ & $\cdots$ & $0$ & $-p$ & $\ldots$ & $-p$ & $s$ &$t_{1,1}$ &$\cdots$ &$t_{1,\ell}$ \\
$0$ & $-1$ & $0$ & $\cdots$ & $0$ & $0$ & $\ldots$ & $0$ & $-p+1$ &$t_{2,1}$ &$\cdots$ &$t_{2,\ell}$\\
\hline
$0$ & $0$ & $0$ & $\cdots$ & $0$ & $0$ & $\cdots$ & $0$ & $0$ &$t_{3,1}$ &$\cdots$ &$t_{3,\ell}$\\ 
$\vdots$ & $\vdots$ & $\vdots$ & $\cdots$ & $\vdots$ & $\vdots$ & $\cdots$ & $\vdots$ & $\vdots$ & $\vdots$ & $\cdots$ & $\vdots$\\
$0$ & $0$ & $0$ & $\cdots$ & $0$ & $0$ & $\cdots$ & $0$ & $0$  &$t_{m,1}$ &$\cdots$ &$t_{m,\ell}$
\end{tabular}\right]$}
\end{center}
reduces (under operations invertible over $\bZ$) to
$$\left[\begin{tabular}{cc|ccc|c|ccc}
$1$ & $0$ & $0$ & $\cdots$ & $0$ & $0$ & $0$ & $\cdots$ &$0$ \\
$0$ & $ap$ &$0$& $\cdots$ & $0$ & $0$ & $0$ & $\cdots$ &$0$ \\
\hline
$0$ & $0$ &  \multicolumn{3}{c|}{\multirow{3}{*}{$p\mathbb{I}_{x+y}$}}&$0$& $0$ & $\cdots$ &$0$\\
$\vdots$ & $\vdots$ &    \multicolumn{3}{c|}{}&$\vdots$& $\vdots$ & $\ddots$ &$\vdots$\\
$0$ & $0$ &    \multicolumn{3}{c|}{}&$0$& $0$ & $\cdots$ &$0$\\
\hline
$0$ & $0$ & $0$ & $\cdots$ & $0$ & $p$&$t_{1,1}$ &$\cdots$ &$t_{1,\ell}$ \\
$0$ & $0$ & $0$ & $\cdots$ & $0$ & $-p$&$t_{2,1}$ &$\cdots$ &$t_{2,\ell}$\\
\hline
$0$ & $0$ & $0$ & $\cdots$ & $0$ & $0$&$t_{3,1}$ &$\cdots$ &$t_{3,\ell}$\\
$\vdots$ & $\vdots$ & $\vdots$ & $\ddots$ & $\vdots$ & $\vdots$ & $\vdots$ & $\ddots$ & $\vdots$\\
$0$ & $0$ & $0$ & $\cdots$ & $0$ & $0$&$t_{m,1}$ &$\cdots$ &$t_{m,\ell}$
\end{tabular}\right]$$
\end{lemma}

\begin{proof}
For $1\leq i\leq x$ and $1\leq j\leq y$ add Row $2+i$ and Row $2+x+j$ to Row $x+y+3$ of $L$. Subsequently, for $1\leq i\leq x$, add Column $2+i$ to Column $2$ and, for $1\leq j\leq y$, Column $2+x+j$ to Column $3+x+y$. Next add Column $2$ to Column $1$.

Now add Column $1$ to Column $3+x+y$ and $p-1$ copies of Column $1$ to Column $2$. Row $1$ can now be used to clear all non-zeros from Column $1$. Once this is completed it is easy to see the that the remaining non-zeros in Column $2$ can also be cleared.
\end{proof}

The proof of Lemma \ref{lem:composites} is essentially a special case of the proof of Theorem \ref{thm:primes_and_composites}, however, to aid the reader, we detail this simpler case before proving the general result.

\begin{lemma}
\label{lem:composites}
Let $1\leq k$ and let $2\leq m,a_1,a_2, \ldots, a_k$. Then there exists a spherical latin bitrade whose canonical group is isomorphic to $\bigoplus_{i=1}^k \bZ_{m a_i}.$
\end{lemma}

\begin{proof}
We begin by defining a digraph $D_{m;a_1,a_2,\ldots,a_k}$ with  vertex set $\{\alpha_0, \alpha_1, \alpha_2,\ldots, \alpha_{k},\break\gamma_1,\gamma_2,\ldots, \gamma_k\}$ and 
\begin{itemize}[noitemsep]
\item for each $1\leq i\leq k$:
\begin{itemize}[noitemsep]
\item[$\circ$] $m-1$ arcs from $\alpha_i$ to $\gamma_i$ and $m-1$ arcs from $\gamma_i$ to $\alpha_i$;
\item[$\circ$] $a_i-1$ arcs from $\alpha_{i-1}$ to $\gamma_i$ and $a_i-1$ arcs from $\gamma_i$ to $\alpha_{i-1}$;
\item[$\circ$] an arc from $\alpha_{i}$ to $\alpha_{i-1}$; 
\end{itemize}
\item for each $1\leq i\leq k -1$: an arc from $\gamma_{i}$ to $\gamma_{i+1}$;  and
\item an additional arc from $\alpha_{0}$ to $\gamma_1$ and an additional arc from $\gamma_k$ to $\alpha_k$.
\end{itemize}
The digraph $D_{m;a_1,a_2,\ldots,a_k}$ has a directed Eulerian spherical embedding and satisfies the connectivity conditions of Proposition \ref{prop:simple}, as can be seen from Figure \ref{fig:Dma_embedding} (in this figure $t$ arcs from $u$ to $v$ alternating with $t$ arcs from $v$ to $u$ are represented by a bidirectional edge labelled $t$). Hence, there exists a spherical latin bitrade whose canonical group is isomorphic to $\cS(D_{m;a_1,a_2,\ldots,a_k})$.

\begin{figure}[!h]
\begin{center}
{\begin{tikzpicture}[fill=gray!50, scale=1.1, vertex/.style={circle,inner sep=2,fill=black,draw}, dot/.style={circle,inner sep=0.7,fill=black,draw}]

\coordinate (a0) at (0.5,1);
\coordinate (a1) at (2,0);
\coordinate (a2) at (4,0);
\coordinate (a3) at (6,0);
\coordinate (ak2) at (8.5,0);
\coordinate (ak1) at (10.5,0);
\coordinate (ak) at (12.5,0);

\coordinate (c1) at (2,2);
\coordinate (c2) at (4,2);
\coordinate (c3) at (6,2);
\coordinate (ck2) at (8.5,2);
\coordinate (ck1) at (10.5,2);
\coordinate (ck) at (12.5,2);


\draw[thick,->-=.55, -<-=.45] (a0) -- (c1);
\draw[thick,->-=.55, -<-=.45] (a1) -- (c2);
\draw[thick,->-=.55, -<-=.45] (a2) -- (c3);
\draw[thick,->-=.55, -<-=.45] (ak2) -- (ck1);
\draw[thick,->-=.55, -<-=.45] (ak1) -- (ck);

\draw[thick,->-=.55, -<-=.45] (a1) -- (c1);
\draw[thick,->-=.55, -<-=.45] (a2) -- (c2);
\draw[thick,->-=.55, -<-=.45] (a3) -- (c3);
\draw[thick,->-=.55, -<-=.45] (ak2) -- (ck2);
\draw[thick,->-=.55, -<-=.45] (ak1) -- (ck1);
\draw[thick,->-=.55, -<-=.45] (ak) -- (ck);

\draw[thick,->-=.55, -<-=.45] (a0) -- (c1);

\draw[thick,->-=.5] (ak) -- (ak1);
\draw[thick,->-=.5] (ak1) -- (ak2);
\draw[thick,->-=.5] (a3) -- (a2);
\draw[thick,->-=.5] (a2) -- (a1);
\draw[thick,->-=.5] (a1) -- (a0);

\draw[thick,->-=.5] (c1) -- (c2);
\draw[thick,->-=.5] (c2) -- (c3);
\draw[thick,->-=.5] (ck2) -- (ck1);
\draw[thick,->-=.5] (ck1) -- (ck);

\draw[thick] (a3) -- (6.5,0);
\draw[thick] (c3) -- (6.5,2);

\draw[thick] (ak2) -- (8,0);
\draw[thick] (ck2) -- (8,2);

\draw [thick, ->-=.5] (a0) to [bend left=60](c1);
\draw [thick, ->-=.5] (ck) to [bend left=90](ak);


\node at (7,0) [dot]{};
\node at (7.25,0) [dot]{};
\node at (7.5,0) [dot]{};

\node at (7,2) [dot]{};
\node at (7.25,2) [dot]{};
\node at (7.5,2) [dot]{};


\node at (a0) [vertex,label=south:$\alpha_0$]{};
\node at (a1) [vertex,label=south:$\alpha_1$]{};
\node at (a2) [vertex,label=south:$\alpha_2$]{};
\node at (a3) [vertex,label=south:$\alpha_3$]{};
\node at (ak2) [vertex,label=south:$\alpha_{k-2}$]{};
\node at (ak1) [vertex,label=south:$\alpha_{k-1}$]{};
\node at (ak) [vertex,label=south:$\alpha_k$]{};

\node at (c1) [vertex,label=north:$\gamma_1$]{};
\node at (c2) [vertex,label=north:$\gamma_2$]{};
\node at (c3) [vertex,label=north:$\gamma_3$]{};
\node at (ck2) [vertex,label=north:$\gamma_{k-2}$]{};
\node at (ck1) [vertex,label=north:$\gamma_{k-1}$]{};
\node at (ck) [vertex,label=north:$\gamma_k$]{};


\node at (1.6,0.95){\tiny $m-1$};
\node at (3.6,0.7){\tiny $m-1$};
\node at (5.6,0.7){\tiny $m-1$};
\node at (8.1,0.95){\tiny $m-1$};
\node at (10.1,0.7){\tiny $m-1$};
\node at (12.1,0.7){\tiny $m-1$};

\node at (2.8,1.2){\tiny $a_2-1$};
\node at (4.8,1.2){\tiny $a_3-1$};
\node at (9.2,1.2){\tiny $a_{k-1}-1$};
\node at (11.3,1.2){\tiny $a_k-1$};

\node at (1.6,1.35){\tiny $a_1-1$};

\end{tikzpicture}}
\end{center}

\caption{A directed Eulerian spherical embedding of $D_{m;a_1,a_2,\ldots,a_k}$.}
\label{fig:Dma_embedding}
\end{figure}
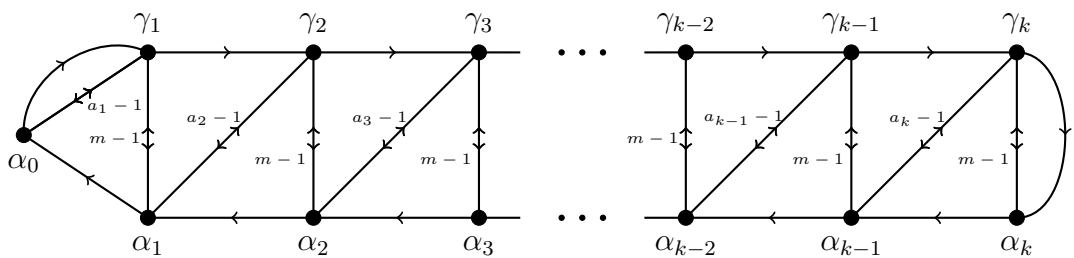

Suppose that we order the vertices of $D_{m;a_1,a_2,\ldots,a_k}$ by $\alpha_k, \gamma_k, \alpha_{k-1}, \gamma_{k-1}, \ldots, \alpha_2, \gamma_2,\break \alpha_1, \gamma_1, \alpha_0$, and construct the associated asymmetric Laplacian. Then, removing the row and column corresponding to $\alpha_0$ yields the reduced asymmetric Laplacian $\cL'(D_{m;a_1,a_2,\ldots,a_k})$.

Let $k\geq 1$ and $m,a_1,a_2,\ldots,a_{k+1}\geq 2$.
Note that 
{\footnotesize $\cL'(D_{m;a_1})=\begin{bmatrix}
m & -m+1 \\
-m & m+a_1-1
\end{bmatrix}$} reduces  to {\footnotesize $\begin{bmatrix}
1 & 0 \\
0 & ma_1
\end{bmatrix}$}; so $\cS(D_{m;a_1})\cong \bZ_{m a_1}$.

Assume that $\cS(D_{m;a_1,a_2,\ldots,a_{k}})$ is isomorphic to $\bigoplus_{i=1}^{k} \bZ_{m a_i}$.  Setting $a_i-1=a_i'$ for $1\leq i\leq k$,
the reduced asymmetric Laplacian $\cL'_k=\cL'(D_{m;a_1,a_2,\ldots,a_{k}})$ is shown below.

\vspace{-5mm}

\begin{center}
\resizebox{\linewidth}{!}{$\cL'_k=\begin{bmatrix}
m & -m+1 & -1 & 0 & 0 & 0 & \cdots & 0 & 0\\
-m & m+a_k' & -a_k' & 0 & 0 & 0 & \cdots & 0 & 0\\
0 & -a_k' & m+a_k' & -m+1 & -1 & 0 & \cdots & 0 & 0\\
0 & -1 & -m+1 & m+a_{k-1}' & -a_{k-1}' & 0 & \cdots & 0 & 0\\ 
0 & 0 & 0 & -a_{k-1}' & m+a_{k-1}' & -m+1 & \cdots & 0 & 0\\ 
0 & 0 & 0 & -1 & -m+1 & m+a_{k-2}' & \cdots & 0 & 0\\ 
\vdots & \vdots & \vdots & \vdots & \vdots & \vdots & \ddots & \vdots & \vdots\\
0 & 0 & 0 & 0 & 0 & 0 & \cdots & m+a_2' & -m+1\\
0& 0 & 0 & 0 & 0 & 0 & \cdots & -m+1 & m+a_1'
\end{bmatrix}$}
\end{center}

Now, consider the digraph $D_{m;a_1,a_2,\ldots,a_{k+1}}$. Applying Lemma \ref{cl:prime+comps}, with $p=m$ and $x=y=0$, to rows $\alpha_{k+1}, \gamma_{k+1},\alpha_{k}, \gamma_{k}$ we have that $\cL'_{k+1}=\cL'(D_{m;a_1,a_2,\ldots,a_{k+1}})$ reduces to
$$\left[\begin{tabular}{cc|ccc}
$1$ & $0$ & $0$ & $\cdots$ & $0$\\
$0$ & $m a_{k+1}$ & $0$ & $\cdots$ & $0$\\
\hline
\vspace{-2mm}
$0$ & $0$ & \\
$\vdots$ & $\vdots$ & &$\cL'_k$&\\
 $0$ & $0$ & && 
\end{tabular}\right].$$
It follows that $\cS(D_{m;a_1,a_2,\ldots,a_{k+1}})$ is isomorphic to $\bigoplus_{i=1}^{k+1} \bZ_{m a_i}$.
\end{proof}

It is now easy to establish the existence of spherical latin bitrades whose canonical groups can be expressed as the direct sum of components of composite order.

\begin{theorem}
\label{thm:composites}
Suppose that $\Gamma$ is a group isomorphic to a direct sum of cyclic groups of composite order; i.e. $\Gamma$ is isomorphic to $\oplus^k_{i=1} \bZ_{n_i}$, where each $n_i$ is composite. Then there exists a spherical latin bitrade whose canonical group is isomorphic to $\Gamma$.
\end{theorem}

\begin{proof}
Let $n_1, n_2,\ldots, n_k$ be composite integers and consider $\Gamma\cong\oplus^k_{i=1} \mathbb{Z}_{n_i}$. 
Recall that if $\gcd(n_u,n_v)=1$, $u\neq v$, then $\oplus^k_{i=1} \bZ_{n_i}\cong \bZ_{n_1}\oplus \cdots \oplus \bZ_{n_{u-1}}\oplus\bZ_{n_{u+1}}\oplus \cdots \oplus \bZ_{n_{v-1}}\oplus\bZ_{n_{v+1}}\oplus \cdots \oplus \bZ_{n_k}\oplus \bZ_{n_un_v}$. Thus we  may assume that $\gcd\{n_1,n_2,\ldots, n_k\}\neq 1$. 
Hence there exists a prime, $p$ say, such that $p$ divides $\gcd\{n_1,n_2,\ldots, n_k\}$. Note that, as $n_i$ is composite for all $1\leq i\leq k$, $p\neq n_i$. By setting $m=p$ and applying Lemma \ref{lem:composites} the result follows.
\end{proof}

The next result addresses the existence of spherical latin bitrades for which the Smith Normal Form of their canonical groups contains components of prime order.

\begin{theorem}
\label{thm:primes_and_composites}
Let $p$ be a prime and let $2\leq a_1,a_2, \ldots, a_k$. Further let $n\leq 1+2\sum_{i=1}^k (a_i-1)$. Then there exists spherical latin bitrade whose canonical group is isomorphic to 
$$\bZ_p^n\oplus\left(\bigoplus_{i=1}^k \bZ_{p a_i}\right).$$
\end{theorem}

\begin{proof}
If $n=0$, then this is Lemma \ref{lem:composites}. So for the remainder of the proof assume that $n\geq 1$. As $n\leq 1+2\sum_{i=1}^k (a_i-1)$ there exists a $k'$, $0\leq k'< k$, and $t$, $0\leq t\leq 2a_{k'+1}-1$ such that 
$$n=1+2\sum_{i=1}^{k'} (a_i-1)+t.$$

First construct the graph $D_{p;a_1,a_2,\ldots,a_k}$ (from the proof of Lemma \ref{lem:composites}). Next, add the following vertices,
\begin{itemize}[noitemsep]
\item for each $1\leq i\leq k'$: add vertices $\delta_{i,j}$ and $\epsilon_{i,j}$ for all $1\leq i\leq a_i-1$;
\item for each $1\leq j\leq \lceil t/2\rceil$: add vertices $\delta_{k'+1,j}$;
\item for each $1\leq j\leq \lfloor t/2\rfloor$: add vertices $\epsilon_{k'+1,j}$; and
\item the vertex $\epsilon_{1,0}$.
\end{itemize}
Now,
\begin{itemize}[noitemsep]
\item replace an arc from $\alpha_{0}$ to $\gamma_1$ with a single arc from  $\epsilon_{1,0}$ to $\gamma_1$ and $p-1$ arcs from $\epsilon_{1,0}$ to $\alpha_{i-1}$ and $p$ arcs from $\alpha_{0}$ to $\epsilon_{1,0}$.
\item for each $1\leq i\leq k'$:
\begin{itemize}[noitemsep]
\item[$\circ$] replace the arcs from $\gamma_i$ to $\alpha_{i-1}$ with single arcs from  $\delta_{i,j}$ to $\alpha_{i-1}$, $p-1$ arcs from $\delta_{i,j}$ to $\gamma_i$ and $p$ arcs from $\gamma_i$ to $\delta_{i,j}$, where $1\leq j\leq a_i-1$.
\item[$\circ$] replace the arcs from $\alpha_{i-1}$ to $\gamma_i$ with single arcs from  $\epsilon_{i,j}$ to $\gamma_i$, $p-1$ arcs from $\epsilon_{i,j}$ to $\alpha_{i-1}$ and $p$ arcs from $\alpha_{i-1}$ to $\epsilon_{i,j}$, where $1\leq j\leq a_i-1$.
\end{itemize}
\item replace $\lceil t/2\rceil$ arcs from $\gamma_{k'+1}$ to $\alpha_{k'}$ with single arcs from  $\delta_{k'+1,j}$ to $\alpha_{k'}$, $p-1$ arcs from $\delta_{k'+1,j}$ to $\gamma_{k'+1}$ and $p$ arcs from $\gamma_{k'+1}$ to $\delta_{k'+1,j}$, where $1\leq j\leq \lceil t/2\rceil$.
\item replace $\lfloor t/2\rfloor$ arcs from $\alpha_{k'}$ to $\gamma_{k'+1}$ with single arcs from  $\epsilon_{k'+1,j}$ to $\gamma_{k'+1}$, $p-1$ arcs from $\epsilon_{k'+1,j}$ to $\alpha_{k'}$ and $p$ arcs from $\alpha_{k'}$ to $\epsilon_{k'+1,j}$, where $1\leq j\leq \lfloor t/2\rfloor$.
\end{itemize}

Call the resulting digraph $D^n_{p;a_1,a_2,\ldots,a_k}$, see Figure \ref{fig:Dman_construction} for an illustration of its construction. 
Note that $D^n_{p;a_1,a_2,\ldots,a_k}$ has a directed spherical embedding, and that it satisfies the connectivity conditions of Proposition \ref{prop:simple}. Therefore, there exists a spherical latin bitrade whose canonical group is isomorphic to $\cS(D^n_{p;a_1,a_2,\ldots,a_k})$.

\begin{figure}
For $i\leq k'$:
\begin{center}
\scalebox{0.85}
{\begin{tikzpicture}[fill=gray!50, scale=1.1, vertex/.style={circle,inner sep=2,fill=black,draw}, dot/.style={circle,inner sep=0.7,fill=black,draw}]

\coordinate (ai1) at (0,1);
\coordinate (ai) at (2,1);
\coordinate (ci1) at (0,3);
\coordinate (ci) at (2,3);

\coordinate (bi1) at (6,0);
\coordinate (bi) at (10,0);
\coordinate (di1) at (6,4);
\coordinate (di) at (10,4);

\coordinate (de1) at (6.75,3.25);
\coordinate (def) at (8.5,1.5);

\coordinate (ep1) at (7.5,2.5);
\coordinate (epf) at (9.25,0.75);


\draw[thick,->-=.55, -<-=.45] (ai1) -- (ci1);
\draw[thick,->-=.55, -<-=.45] (ai) -- (ci);
\draw[thick,->-=.55, -<-=.45] (ai1) -- (ci);

\draw[thick,->-=.5] (ai) -- (ai1);
\draw[thick,->-=.5] (ci1) -- (ci);


\draw[thick,->-=.55, -<-=.45] (bi1) -- (di1);
\draw[thick,->-=.55, -<-=.45] (bi) -- (di);

\draw[thick,->-=.5] (bi) to [bend left=20] (bi1);
\draw[thick,->-=.5] (di1) to [bend left=20] (di);

\draw[thick,->-=.5] (de1) -- (bi1);
\draw[thick,->-=.5] (def) -- (bi1);

\draw[thick,->-=.5] (ep1) -- (di);
\draw[thick,->-=.5] (epf) -- (di);

\draw [thick,->-=.55, -<-=.45] (bi1) to [bend right=10](ep1);
\draw [thick,->-=.5] (bi1) to [bend left=10](ep1);

\draw [thick,->-=.55, -<-=.45] (bi1) to [bend right=10](epf);
\draw [thick,->-=.5] (bi1) to [bend left=10](epf);

\draw [thick,->-=.55, -<-=.45] (di) to [bend right=10](de1);
\draw [thick,->-=.5] (di) to [bend left=10](de1);

\draw [thick,->-=.55, -<-=.45] (di) to [bend right=10](def);
\draw [thick,->-=.5] (di) to [bend left=10](def);

\draw [ultra thick, ->] (3,2) -- (4.5,2);

\node at (7.9,2.1) [dot]{};
\node at (8,2) [dot]{};
\node at (8.1,1.9) [dot]{};


\node at (ai1) [vertex,label=south:$\alpha_{i-1}$]{};
\node at (ai) [vertex,label=south:$\alpha_{i}$]{};

\node at (bi1) [vertex,label=south:$\alpha_{i-1}$]{};
\node at (bi) [vertex,label=south:$\alpha_{i}$]{};

\node at (ci1) [vertex,label=north:$\gamma_{i-1}$]{};
\node at (ci) [vertex,label=north:$\gamma_{i}$]{};

\node at (di1) [vertex,label=north:$\gamma_{i-1}$]{};
\node at (di) [vertex,label=north:$\gamma_{i}$]{};

\node at (de1) [vertex]{};
\node at (6.5,3.5) {\small $\delta_{i,1}$};
\node at (def) [vertex]{};
\node at (8.8,1.15) {\small $\delta_{i,a_i-1}$};

\node at (ep1) [vertex]{};
\node at (7.25,2.75) {\small $\epsilon_{i,1}$};
\node at (epf) [vertex]{};
\node at (9.5,0.5) {\small $\epsilon_{i,a_i-1}$};


\node at (0.67,2.1){\tiny $a_i-1$};
\node at (-0.4,1.95){\tiny $p-1$};
\node at (2.4,1.95){\tiny $p-1$};

\node at (8.2,3.95){\tiny $p-1$};
\node at (8.7,2.75){\tiny $p-1$};

\node at (7.8,0.05){\tiny $p-1$};
\node at (7.3,1.25){\tiny $p-1$};

\node at (5.6,1.95){\tiny $p-1$};
\node at (10.4,1.95){\tiny $p-1$};

\end{tikzpicture}}
\end{center}


Let $a=a_k'-1$, then, if $t=2\ell$:
\begin{center}
\scalebox{0.85}
{\begin{tikzpicture}[fill=gray!50, scale=1.1, vertex/.style={circle,inner sep=2,fill=black,draw}, dot/.style={circle,inner sep=0.7,fill=black,draw}]

\coordinate (ak) at (0,1);
\coordinate (ak1) at (2,1);
\coordinate (ck) at (0,3);
\coordinate (ck1) at (2,3);

\coordinate (bk) at (6,0);
\coordinate (bk1) at (10,0);
\coordinate (dk) at (6,4);
\coordinate (dk1) at (10,4);

\coordinate (de1) at (6.75,3.25);

\coordinate (epl) at (8,2);


\draw[thick,->-=.55, -<-=.45] (ak) -- (ck);
\draw[thick,->-=.55, -<-=.45] (ak1) -- (ck1);
\draw[thick,->-=.55, -<-=.45] (ak) -- (ck1);

\draw[thick,->-=.5] (ak1) -- (ak);
\draw[thick,->-=.5] (ck) -- (ck1);


\draw[thick,->-=.55, -<-=.45] (bk) -- (dk);
\draw[thick,->-=.55, -<-=.45] (bk1) -- (dk1);

\draw[thick,->-=.5] (bk1) to [bend left=20] (bk);
\draw[thick,->-=.5] (dk) to [bend left=20] (dk1);

\draw[thick,->-=.5] (de1) -- (bk);

\draw[thick,->-=.5] (epl) -- (dk1);

\draw [thick,->-=.55, -<-=.45] (bk) to [bend right=10](epl);
\draw [thick,->-=.5] (bk) to [bend left=10](epl);

\draw [thick,->-=.55, -<-=.45] (dk1) to [bend right=10](de1);
\draw [thick,->-=.5] (dk1) to [bend left=10](de1);

\draw [ultra thick, ->] (3,2) -- (4.5,2);

\draw [thick,->-=.55, -<-=.45] (dk1) to [bend left=30](bk);

\node at (7.275,2.725) [dot]{};
\node at (7.375,2.625) [dot]{};
\node at (7.475,2.525) [dot]{};


\node at (ak) [vertex,label=south:$\alpha_{k'}$]{};
\node at (ak1) [vertex,label=south:$\alpha_{k'+1}$]{};

\node at (bk) [vertex,label=south:$\alpha_{k'}$]{};
\node at (bk1) [vertex,label=south:$\alpha_{k'+1}$]{};

\node at (ck) [vertex,label=north:$\gamma_{k'}$]{};
\node at (ck1) [vertex,label=north:$\gamma_{k'+1}$]{};

\node at (dk) [vertex,label=north:$\gamma_{k'}$]{};
\node at (dk1) [vertex,label=north:$\gamma_{k'+1}$]{};

\node at (de1) [vertex]{};
\node at (6.5,3.5) {\small $\delta_{k'+1,1}$};

\node at (epl) [vertex]{};
\node at (7.3,2) {\small $\epsilon_{k'+1,\ell}$};


\node at (0.75,2.1){\tiny $a$};
\node at (-0.4,1.95){\tiny $p-1$};
\node at (2.4,1.95){\tiny $p-1$};

\node at (8.2,3.95){\tiny $p-1$};
\node at (7.4,0.8){\tiny $p-1$};

\node at (5.6,1.95){\tiny $p-1$};
\node at (10.4,1.95){\tiny $p-1$};

\node at (9.125,1.35) {\tiny $a-\ell$};

\end{tikzpicture}}
\end{center}


Again let $a=a'_k-1$, then, if $t=2\ell +1$:
\begin{center}
\scalebox{0.85}
{\begin{tikzpicture}[fill=gray!50, scale=1.1, vertex/.style={circle,inner sep=2,fill=black,draw}, dot/.style={circle,inner sep=0.7,fill=black,draw}]

\coordinate (ak) at (0,1);
\coordinate (ak1) at (2,1);
\coordinate (ck) at (0,3);
\coordinate (ck1) at (2,3);

\coordinate (bk) at (6,0);
\coordinate (bk1) at (10.5,-0.5);
\coordinate (dk) at (6,4);
\coordinate (dk1) at (10,4);

\coordinate (de1) at (6.75,3.25);

\coordinate (epl) at (7.8,2.2);

\coordinate (del1) at (8.6,1.4);


\draw[thick,->-=.55, -<-=.45] (ak) -- (ck);
\draw[thick,->-=.55, -<-=.45] (ak1) -- (ck1);
\draw[thick,->-=.55, -<-=.45] (ak) -- (ck1);

\draw[thick,->-=.5] (ak1) -- (ak);
\draw[thick,->-=.5] (ck) -- (ck1);


\draw[thick,->-=.55, -<-=.45] (bk) -- (dk);
\draw[thick,->-=.55, -<-=.45] (bk1) to [bend right=20] (dk1);

\draw[thick,->-=.5] (bk1) to [bend left=20] (bk);
\draw[thick,->-=.5] (dk) to [bend left=20] (dk1);

\draw[thick,->-=.5] (de1) -- (bk);

\draw[thick,->-=.5] (epl) -- (dk1);

\draw [thick,->-=.55, -<-=.45] (bk) to [bend right=10](epl);
\draw [thick,->-=.5] (bk) to [bend left=10](epl);

\draw [thick,->-=.55, -<-=.45] (dk1) to [bend right=10](de1);
\draw [thick,->-=.5] (dk1) to [bend left=10](de1);

\draw[thick,->-=.5] (del1) -- (bk);
\draw [thick,->-=.55, -<-=.45] (dk1) to [bend right=10](del1);
\draw [thick,->-=.5] (dk1) to [bend left=10](del1);

\draw [ultra thick, ->] (3,2) -- (4.5,2);

\draw [thick,-<-=.5] (dk1) to [out=275, in=45] (9.5,0.5) to [out=225, in=355] (bk);

\draw [thick,->-=.55, -<-=.45] (dk1) to [out=290, in=45] (9.75,0.25) to [out=225, in=340] (bk);

\node at (7.175,2.825) [dot]{};
\node at (7.275,2.725) [dot]{};
\node at (7.375,2.625) [dot]{};


\node at (ak) [vertex,label=south:$\alpha_{k'}$]{};
\node at (ak1) [vertex,label=south:$\alpha_{k'+1}$]{};

\node at (bk) [vertex,label=south:$\alpha_{k'}$]{};
\node at (bk1) [vertex,label=south:$\alpha_{k'+1}$]{};

\node at (ck) [vertex,label=north:$\gamma_{k'}$]{};
\node at (ck1) [vertex,label=north:$\gamma_{k'+1}$]{};

\node at (dk) [vertex,label=north:$\gamma_{k'}$]{};
\node at (dk1) [vertex,label=north:$\gamma_{k'+1}$]{};

\node at (de1) [vertex]{};
\node at (6.5,3.5) {\small $\delta_{k'+1,1}$};

\node at (del1) [vertex]{};
\node at (9.1,1.05) {\small $\delta_{k'+1,\ell+1}$};

\node at (epl) [vertex]{};
\node at (7.2,2.2) {\small $\epsilon_{k'+1,\ell}$};


\node at (0.75,2.1){\tiny $a$};
\node at (-0.4,1.95){\tiny $p-1$};
\node at (2.4,1.95){\tiny $p-1$};

\node at (8.2,3.95){\tiny $p-1$};
\node at (7.5,1.1){\tiny $p-1$};

\node at (8.7,2.5){\tiny $p-1$};

\node at (5.6,1.95){\tiny $p-1$};
\node at (11.1,1.8){\tiny $p-1$};

\node at (10.1,0) {\tiny $a-\ell-1$};

\end{tikzpicture}}
\end{center}

\caption{Constructing $D^n_{m;a_1,a_2,\ldots,a_k}$.}
\label{fig:Dman_construction}
\end{figure}

For ease of notation, let
$$d_i=\left\{\begin{array}{ll}
a_i-1 & \text{for }1\leq i\leq k'\\
\lceil t/2\rceil & \text{for }i=k'+1\\
0&\text{otherwise}
\end{array}\right.
\quad\text{ and }\quad
e_i=\left\{\begin{array}{ll}
a_i-1 & \text{for }1\leq i\leq k'\\
\lfloor t/2\rfloor & \text{for }i=k'+1\\
0&\text{otherwise}
\end{array}\right..$$
Suppose that we order the vertices of $D^n_{p;a_1,a_2,\ldots,a_k}$ by 
$$(\alpha_k,\gamma_k,\delta_{k,d_k},\ldots,\delta_{k,1},\epsilon_{k,e_k}, \ldots, \epsilon_{k,1}), \ldots, 
(\alpha_2,\gamma_2,\delta_{2,d_{2}},\ldots,\delta_{2,1},\epsilon_{2,e_{2}}, \ldots, \epsilon_{2,1}),$$
$$(\alpha_1,\gamma_1,\delta_{2,d_{1}},\ldots,\delta_{1,1},\epsilon_{1,e_{1}}, \ldots, \epsilon_{1,1}, \epsilon_{1,0}),\alpha_0$$
and construct the associated asymmetric Laplacian. Then, removing the row and column corresponding to $\alpha_0$ yields the reduced asymmetric Laplacian $\cL'(D^n_{p;a_1,a_2,\ldots,a_k})$.

Let $k\geq 1$ and $p,a_1,a_2,\ldots,a_{k+1}\geq 2$ and let $1\leq n\leq 1+2\sum_{i=1}^{k+1} (a_i-1)$. Then, letting $x=d_1$, $y=e_1$ and $r=p(x+1)-a_1-x-1$, 
$$\cL'\left(D^{\min\{n,1+2(a_1-1)\}}_{p;a_1}\right)=\left[
\begin{tabular}{cc|ccc|ccc|c}
$p$ & $-p+1$ & $0$ & $\cdots$ & $0$ & $0$ & $\ldots$ & $0$ & $0$ \\
$-p$ & $r$ & $-p$ &  $\ldots$ & $-p$ & $0$ & $\ldots$ & $0$ & $0$ \\
\hline
$0$ & $-p+1$ &  \multicolumn{3}{c|}{\multirow{3}{*}{$p\mathbb{I}_{d_1}$}}&$0$ & $\cdots$ & $0$ & $0$\\
$\vdots$ & $\vdots$ & \multicolumn{3}{c|}{}&$\vdots$ &$\ddots$ &$\vdots$&$\vdots$ \\
$0$ & $-p+1$ & \multicolumn{3}{c|}{}&$0$ & $\cdots$ & $0$ & $0$\\
\hline
$0$ & $-1$ & $0$ & $\cdots$ & $0$ & \multicolumn{3}{c|}{\multirow{3}{*}{$p\mathbb{I}_{e_1}$}} & $0$\\
$\vdots$ & $\vdots$ & $\vdots$ & $\ddots$ & $\vdots$ & \multicolumn{3}{c|}{} 
 & $0$ \\
$0$ & $-1$ & $0$ & $\ldots$ & $0$ & \multicolumn{3}{c|}{} & $\vdots$\\
\hline
$0$ & $-1$ & $0$ & $\ldots$ & $0$ & $0$ & $\cdots$ & $0$ & $p$\\
\end{tabular}
\right].$$
Which reduces, under a similar argument to that used to prove Lemma \ref{cl:prime+comps}, to 
$$\left[
\begin{tabular}{cc|ccc}
$1$ & $0$ & $0$ & $\cdots$ & $0$\\
$0$ & $pa_1$ & $0$ & $\cdots$ & $0$ \\
\hline
$0$ & $0$ &  \multicolumn{3}{c}{\multirow{3}{*}{$p\mathbb{I}_{d_1+e_1+1}$}}\\
$\vdots$ & $\vdots$ & \multicolumn{3}{c}{}\\
$0$ & $0$ & \multicolumn{3}{c}{}
\end{tabular}
\right].$$
Hence, $\cS(D^{\min\{n,1+2(a_1-1)\}}_{p;a_1})\cong\bZ_p^{\min\{n,1+2(a_1-1)\}}\oplus\bZ_{p a_1}$.

Assume that $\cS\left(D^{\min\{n,1+2\sum_{i=1}^k(a_i-1)\}}_{p;a_1,a_2,\ldots,a_k}\right)\cong\bZ_p^{\min\{n,1+2\sum_{i=1}^k(a_i-1)\}}\oplus\left(\bigoplus_{i=1}^k \bZ_{p a_i}\right)$. Denote $\cL'\left(D^{\min\{n,1+2\sum_{i=1}^k(a_i-1)\}}_{p;a_1,a_2,\ldots,a_k}\right)$ by $\cL_k'$ and
consider $\cL'\left(D^n_{p;a_1,a_2,\ldots,a_{k+1}}\right)$. Applying Lemma \ref{cl:prime+comps}, with 
$$x=\left\lceil \frac{1}{2}\max\left\{n-1-2\sum_{i=1}^k(a_i-1),0\right\}\right\rceil$$
and 
 $$y=\left\lfloor \frac{1}{2}\max\left\{n-1-2\sum_{i=1}^k(a_i-1),0\right\}\right\rfloor$$
to rows  $\alpha_{k+1},\gamma_{k+1},\delta_{k+1,x},\ldots,\delta_{k+1,1},\epsilon_{k+1,y}, \ldots, \epsilon_{k+1,1},\alpha_k,\gamma_k$ of $\cL_{k+1}(D^n_{p;a_1,a_2,\ldots,a_{k+1}})$ reduces it to
$$\left[\begin{tabular}{cc|ccc|ccc}
$1$ & $0$ & $0$ & $\cdots$ & $0$ & $0$ & $\cdots$ & $0$\\
$0$ & $p a_{k+1}$ & $0$ & $\cdots$ & $0$ & $0$ & $\cdots$ & $0$\\
\hline
$0$ & $0$ &\multicolumn{3}{c|}{\multirow{3}{*}{{$p\mathbb{I}_{x+y}$}}}& $0$ & $\cdots$ & $0$\\
$\vdots$ & $\vdots$ &  \multicolumn{3}{c|}{} &  $\vdots$ & $\ddots$ & $\vdots$\\
$0$ & $0$ & \multicolumn{3}{c|}{} & $0$ & $\cdots$ & $0$\\
\hline
$0$ & $0$ & $0$ & $\cdots$ & $0$ &  \multicolumn{3}{c}{\multirow{3}{*}{{$\mathcal{L}'_k$}}}\\
$\vdots$ & $\vdots$ & $\vdots$ & $\ddots$ & $\vdots$& \multicolumn{3}{c}{}\\
 $0$ & $0$ & $0$ & $\cdots$ & $0$& \multicolumn{3}{c}{}
\end{tabular}\right].$$
Therefore $\cS(D^n_{p;a_1,a_2,\ldots,a_{k+1}})\cong\bZ_p^n\oplus\left(\bigoplus_{i=1}^{k+1} \bZ_{p a_i}\right)$.
\end{proof}

\subsection{Canonical groups of rank two}

In this section we will restrict our attention to canonical groups of rank two. We show that, with one exception and a further three possible exceptions, any finite abelian group of rank two is isomorphic to the canonical group of some spherical latin bitrade.

We will make use of the following elementary lemma.

\begin{lemma}
\label{lem:121-reduction}
Let $2\leq d$, $1\leq x$, $2\leq y$ and $t_{i,j}\in\mathbb{Z}$ for  $1\leq i\leq x$ and $1\leq j\leq y$. Further let $M$ be the $d-1$ by $d$ matrix where 
$$m_{ij}=\left\{\begin{array}{ll}
2&\text{if } i=j\\
-1&\text{if } j=i+1\text{ or }j=i-1\\
0&\text{otherwise}
\end{array}\right..
$$
Then the $d+x-2$ by $d+y-2$ matrix 
$$\left[\begin{tabular}{ccc|ccccc}
\multicolumn{5}{c|}{\multirow{4}{*}{{\large$M$}}}& $0$ & $\cdots$ & $0$ \\
\multicolumn{5}{c|}{}& $0$ & $\cdots$ & $0$ \\
\multicolumn{5}{c|}{}&$\vdots$ &$\ddots$ &$\vdots$ \\
\multicolumn{5}{c|}{}& $0$ & $\cdots$ & $0$ \\
\hline
$0$ & $\cdots$ & $0$ & $t_{1,1}$ & $t_{1,2}$& $t_{1,3}$ & $\cdots$ & $t_{1,y}$\\
$\vdots$ & $\ddots$ & $\vdots$ & $\vdots$ & $\vdots$& $\vdots$ & $\ddots$ & $\vdots$\\
$0$ & $\cdots$ & $0$ & $t_{x,1}$ & $t_{x,2}$& $t_{x,3}$ & $\cdots$ & $t_{x,y}$
\end{tabular}\right]$$
reduces (under operations invertible over $\mathbb{Z}$) to 
$$\left[\begin{tabular}{ccc|ccccc}
\multicolumn{3}{c|}{\multirow{3}{*}{{\large$\mathbb{I}_{d-2}$}}}& $0$& \multicolumn{1}{c|}{$0$} &  $0$ & $\cdots$ & $0$ \\
\multicolumn{3}{c|}{}&$\vdots$&\multicolumn{1}{c|}{$\vdots$}&$\vdots$ &$\ddots$ &$\vdots$ \\
\multicolumn{3}{c|}{}& $0$ & \multicolumn{1}{c|}{$0$}&  $0$ & $\cdots$ & $0$ \\
\cline{1-5}
$0$ & $\cdots$ & $0$& $d$ &  \multicolumn{1}{c|}{$1-d$} & $0$& $\cdots$ & $0$ \\
\hline 
$0$ & $\cdots$ & $0$ & $t_{1,1}$ & $t_{1,2}$& $t_{1,3}$& $\cdots$ & $t_{1,y}$\\
$\vdots$ & $\ddots$ & $\vdots$ & $\vdots$ &  $\vdots$&  $\vdots$ & $\ddots$ & $\vdots$\\
$0$ & $\cdots$ & $0$ & $t_{x,1}$ & $t_{x,2}$&$t_{x,3}$& $\cdots$ & $t_{x,y}$
\end{tabular}\right].$$
\end{lemma}

\begin{proof}
When $d=2$, the result is trivial. 
Assume that the statement holds for $d=k$, and consider $L(k+1)$. Then $L(k+1)$ reduces to 
$$\left[\begin{tabular}{ccc|cccccc}
\multicolumn{3}{c|}{\multirow{3}{*}{{\large$\mathbb{I}_{k-2}$}}}& $0$ &  $0$&  $0$&  $0$ & $\cdots$ & $0$ \\
\multicolumn{3}{c|}{}&$\vdots$&$\vdots$&$\vdots$&$\vdots$ &$\ddots$ &$\vdots$ \\
\multicolumn{3}{c|}{}& $0$ & $0$& $0$& $0$ & $\cdots$ & $0$ \\
\hline
$0$ & $\cdots$ & $0$& $k$ &  $1-k$ & $0$& $0$& $\cdots$ & $0$ \\
$0$ & $\cdots$ & $0$& $-1$ &  $2$ & $-1$& $0$& $\cdots$ & $0$ \\
$0$ & $\cdots$ & $0$& $0$ & $t_{1,1}$ & $t_{1,2}$& $t_{1,3}$& $\cdots$ & $t_{1,y}$\\
$\vdots$ & $\ddots$ & $\vdots$& $\vdots$ & $\vdots$ &  $\vdots$ & $\ddots$ & $\vdots$\\
$0$ & $\cdots$ & $0$ & $0$& $t_{x,1}$ & $t_{x,2}$&$t_{x,3}$& $\cdots$ & $t_{x,y}$
\end{tabular}\right]$$
Adding $k-1$ copies of Row $k$ to Row $k-1$ followed by adding one copy of the updated Row $k-1$ to Row $k$ yields a $1$ in entry $(k-1,k-1)$ and this is now the only non-zero in Column $k-1$. The result follows.
\end{proof}

\begin{lemma}
\label{lem:ab+bc+ac+1}
Suppose that $1\leq a,b,c$.  Then there exists spherical latin bitrade whose canonical group is isomorphic to 
$\bZ_{ab+bc+ac+1}\oplus\bZ_{ab+bc+ac+1}.$
\end{lemma}

\begin{proof}
Without loss of generality we may assume that $1\leq a\leq b\leq c$. Define $D_{a,b,c}$ to be the digraph of order $a+b+c+1$ with vertex set $\{\alpha_1,\alpha_2,\ldots,\alpha_a,\beta_1,\beta_2,\ldots,\beta_b,\gamma_1,\gamma_2,\break\ldots,\gamma_c,\delta\}$ and
\begin{itemize}[noitemsep]
\item for $1\leq i\leq a-1$ an arc from $\alpha_i$ to $\alpha_{i+1}$ and an arc from $\alpha_{i+1}$ to $\alpha_{i}$;
\item for $1\leq i\leq b-1$ an arc from $\beta_i$ to $\beta_{i+1}$ and an arc from $\beta_{i+1}$ to $\beta_{i}$;
\item for $1\leq i\leq c-1$ an arc from $\gamma_i$ to $\gamma_{i+1}$ and an arc from $\gamma_{i+1}$ to $\gamma_{i}$;
\item for each $\iota\in\{\alpha,\beta,\gamma\}$ an arc from $\delta$ to $\iota_1$ and from $\iota_1$ to $\delta$; and 
\item $a$ arcs from $\beta_b$ to $\gamma_c$ and from $\gamma_c$ to $\beta_b$; $b$ arcs from $\alpha_a$ to $\gamma_c$ and from $\gamma_c$ to $\alpha_a$; and $c$ arcs from $\alpha_a$ to $\beta_b$ and from $\beta_b$ to $\alpha_a$.
\end{itemize}
Note that $D_{a,b,c}$ has a directed Eulerian spherical embedding, see Figure \ref{fig:rank2}, and that $D_{a,b,c}$ satisfies the connectivity conditions of Proposition \ref{prop:simple}. Hence, there exists a spherical latin bitrade whose canonical group is isomorphic to $\cS(D_{a,b,c})$.

\begin{figure}[!h]

\begin{center}
\scalebox{0.9}
{\begin{tikzpicture}[fill=gray!50, scale=1.1, vertex/.style={circle,inner sep=2,fill=black,draw}, dot/.style={circle,inner sep=0.7,fill=black,draw}]

\coordinate (aa) at (4,6);
\coordinate (aa1) at (4,5);
\coordinate (a1) at (4,3);

\coordinate (bb) at (0,0);
\coordinate (bb1) at (1,0.5);
\coordinate (b1) at (3,1.5);

\coordinate (cc) at (8,0);
\coordinate (cc1) at (7,0.5);
\coordinate (c1) at (5,1.5);

\coordinate (d) at (4,2);



\draw [thick,->-=.6, -<-=.4] (bb) -- (bb1);
\draw [thick,->-=.6, -<-=.4] (d) -- (b1);

\draw [thick,->-=.6, -<-=.4] (aa) -- (aa1);
\draw [thick,->-=.6, -<-=.4] (d) -- (a1);

\draw [thick,->-=.6, -<-=.4] (cc) -- (cc1);
\draw [thick,->-=.6, -<-=.4] (d) -- (c1);

\draw [thick,->-=.6, -<-=.4] (bb) -- (cc);
\draw [thick,->-=.6, -<-=.4] (aa) to [bend right=20](bb);
\draw [thick,->-=.6, -<-=.4] (cc) to [bend right=20](aa);

\draw [thick] (aa1) -- (4,4.5);
\draw [thick] (a1) -- (4,3.5);

\draw [thick] (bb1) -- (1.5,0.75);
\draw [thick] (b1) -- (2.5,1.25);

\draw [thick] (cc1) -- (6.5,0.75);
\draw [thick] (c1) -- (5.5,1.25);


\node at (aa) [vertex,label=north:$\alpha_{a}$]{};
\node at (aa1) [vertex,label=east:$\alpha_{a-1}$]{};
\node at (a1) [vertex,label=east:$\alpha_{1}$]{};

\node at (bb) [vertex,label=south west :$\beta_{b}$]{};
\node at (bb1) [vertex,label=north:$\beta_{b-1}$]{};
\node at (b1) [vertex,label=north:$\beta_{1}$]{};

\node at (cc) [vertex,label=south east:$\gamma_{c}$]{};
\node at (cc1) [vertex,label=north:$\gamma_{c-1}$]{};
\node at (c1) [vertex,label=north:$\gamma_{1}$]{};

\node at (d) [vertex,label=north east:$\delta$]{};

\node at (4,4.15) [dot]{};
\node at (4,4) [dot]{};
\node at (4,3.85) [dot]{};

\node at (6,1) [dot]{};
\node at (6.2,0.9) [dot]{};
\node at (5.8,1.1) [dot]{};

\node at (2,1) [dot]{};
\node at (1.8,0.9) [dot]{};
\node at (2.2,1.1) [dot]{};

\node at (4,-0.25) {$a$};
\node at (7,3.5) {$b$};
\node at (1,3.5) {$c$};


\end{tikzpicture}}
\end{center}
\caption{A directed Eulerian spherical embedding of  $D_{a,b,c}$.}
\label{fig:rank2}
\end{figure}
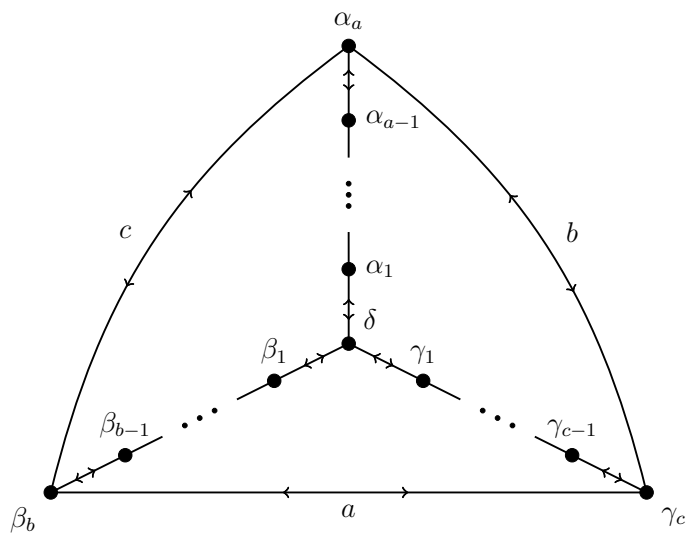

Suppose that we order the vertices of $D_{a,b,c}$ by 
$$\gamma_1,\gamma_2,\ldots,\gamma_{c-2},\gamma_{c-1},\gamma_{c}, \beta_1,\beta_2,\ldots,\beta_{b-2},\beta_{b-1},\beta_b,\alpha_1,\alpha_2,\ldots,\alpha_{a-2},\alpha_{a-1}, \alpha_a,\delta.$$
Let $\cL'(D_{a,b,c})$ be the reduced asymmetric Laplacian for $D_{a,b,c}$ obtained by removing the row and column corresponding to $\delta$.

When $a=b=c=1$, $\cL'(D_{1,1,1})=${\footnotesize$\begin{bmatrix}
3&-1&-1\\
-1&3&-1\\
-1&-1&3
\end{bmatrix}$}, which reduces to {\footnotesize$\begin{bmatrix}
1&0&0\\
0&4&0\\
0&0&4
\end{bmatrix}$}. 

Suppose that $2\leq a,b,c$. Consider $\cL'(D_{a,b,c})$, via three applications of Lemma \ref{lem:ab+bc+ac+1} and setting $a+b+c+1=t$, this reduces to 
$$\left[\begin{tabular}{ccc|cc|ccc|cc|ccc|cc}
\multicolumn{3}{c|}{\multirow{3}{*}{{\large$\mathbb{I}_{c-2}$}}}& $0$& $0$ &  $0$ & $\cdots$ & $0$ & $0$& $0$& $0$ & $\cdots$ & $0$   & $0$& $0$\\\multicolumn{3}{c|}{}&$\vdots$&$\vdots$&$\vdots$ &$\ddots$ &$\vdots$&$\vdots$&$\vdots$&$\vdots$&$\ddots$&$\vdots$&$\vdots$&$\vdots$\\
\multicolumn{3}{c|}{}& $0$ & $0$&  $0$ & $\cdots$ & $0$ &$0$&$0$&$0$& $\cdots$ & $0$   & $0$& $0$\\
\hline
$0$&$\cdots$&$0$&$c$&$1-c$&$0$&$\cdots$&$0$&$0$&$0$&$0$& $\cdots$ & $0$ & $0$& $0$\\
$0$&$\cdots$&$0$&$-1$&$t-c$&$0$&$\cdots$&$0$&$0$&$-a$& $0$ & $\cdots$ & $0$& $0$&$-b$\\
\hline
$0$&$\cdots$&$0$&$0$&$0$&\multicolumn{3}{c|}{\multirow{3}{*}{{\large$\mathbb{I}_{b-2}$}}}& $0$& $0$ &  $0$& $\cdots$ & $0$ & $0$& $0$ \\
$\vdots$&$\ddots$&$\vdots$&$\vdots$&$\vdots$&\multicolumn{3}{c|}{}&$\vdots$&$\vdots$&$\vdots$& $\ddots$ & $\vdots$ & $\vdots$& $\vdots$\\
$0$&$\cdots$&$0$&$0$&$0$&\multicolumn{3}{c|}{}& $0$ & $0$&  $0$& $\cdots$ & $0$ & $0$& $0$ \\
\hline
$0$&$\cdots$&$0$&$0$&$0$&$0$&$\cdots$&$0$&$b$&$1-b$&  $0$ & $\cdots$ & $0$ &$0$&$0$\\
$0$&$\cdots$&$0$&$0$&$-a$&$0$&$\cdots$&$0$&$-1$&$t-b$&  $0$ & $\cdots$ & $0$ &$0$&$-c$\\
\hline
$0$&$\cdots$&$0$&$0$&$0$&  $0$& $\cdots$ & $0$ & $0$& $0$ & \multicolumn{3}{c|}{\multirow{3}{*}{{\large$\mathbb{I}_{a-2}$}}} & $0$& $0$ \\
$\vdots$&$\ddots$&$\vdots$&$\vdots$&$\vdots$& $\vdots$& $\ddots$ & $\vdots$   &$\vdots$&$\vdots$&  \multicolumn{3}{c|}{}  & $\vdots$& $\vdots$\\
$0$&$\cdots$&$0$&$0$&$0$&  $0$& $\cdots$ & $0$  & $0$ & $0$&  \multicolumn{3}{c|}{} & $0$& $0$ \\
\hline
$0$&$\cdots$&$0$&$0$&$0$&$0$&$\cdots$&$0$&$0$&$0$& $0$& $\cdots$ & $0$ & $a$ &$ 1-a$\\
$0$&$\cdots$&$0$&$0$&$-b$&$0$&$\cdots$&$0$&$0$&$-c$&$0$&$\cdots$&$0$&$-1$&$t-a$\\
\end{tabular}\right].$$
Computing the Smith Normal form of 
$$\begin{bmatrix}
c&1-c&0&0&0&0\\
-1&t-c&0&-a&0&-b\\
0&0&b&1-b&0&0\\
0&-a&-1&t-b&0&-c\\
0&0&0&0&a&1-a\\
0&-b&0&-c&-1&t-a
\end{bmatrix}$$
we have that $\cS(D_{a,b,c})\cong \bZ_{ab+bc+ac+1}\oplus\bZ_{ab+bc+ac+1}$. 

The cases where $1=a<b\leq c$ and $1=a=b< c$ follow similarly.
\end{proof}

\begin{theorem}
\label{thm:rank2}
For $n, m\geq 2$, with one exception and a further three possible exceptions, there exists a spherical latin bitrade whose canonical group is isomorphic to $\bZ_{n}\oplus\bZ_{m}$.

The exceptions are as follows. There does not exist a spherical latin bitrade with canonical group isomorphic to $\bZ_{2}\oplus\bZ_{2}$. 
There may or may not exist a spherical latin bitrade with canonical group isomorphic to $\bZ_{3}\oplus\bZ_{3}$ or $\bZ_{5}\oplus\bZ_{5}$ or $\bZ_{r}\oplus\bZ_{r}$ for some $r$ greater than $10^{11}$.

Finally, if we assume the Generalised Riemann Hypothesis, then there exists a spherical latin bitrade with canonical group isomorphic to $\bZ_{r}\oplus\bZ_{r}$.
\end{theorem}

\begin{proof}
If $n$ and $m$ are coprime, then $\bZ_n\oplus\bZ_m\cong\bZ_{nm}$ and the result follows from \cite{CavWan} (it also follows from Lemma \ref{lem:composites} with $k=1$). So assume that $n$ and $m$ are not coprime, that is we are in the rank 2 case.

Suppose that $n\neq m$. If $n$ and $m$ are both composite, then the result follows from Theorem \ref{thm:composites}. So suppose that $n$ is prime and $m$ is composite. Then as $n$ and $m$ are not coprime $m=kn$ for some $k>1$ and the result follows from Theorem \ref{thm:primes_and_composites}.

So, suppose that $n=m$. If there exist $a,b,c\geq 1$ such that $ab+ac+bc+1=n$, then by Lemma \ref{lem:ab+bc+ac+1} there exits a spherical latin bitrade whose canonical group is isomorphic to $\bZ_n\oplus\bZ_n$. In \cite{BorCho} Borwein and Choi proved that there are at most nineteen integers that are not of the form $ab+ac+bc+1$ where $a,b,c\geq 1$. The first eighteen are: $2$, $3$, $5$, $7$, $11$, $19$, $23$, $31$, $43$, $59$, $71$, $79$, $103$, $131$, $191$, $211$, $331$ and $463$. The nineteenth is greater than $10^{11}$ and is not an exception if the Generalised Riemann Hypothesis is assumed. For $n\in\{7,11,19,23,31,43,59,71,79,103,131,191,211,331,463\}$ directed Eulerian spherical embeddings whose underlying digraphs satisfy the connectivity conditions of Proposition \ref{prop:simple} and  with abelian sandpile groups isomorphic to $\bZ_n\oplus\bZ_n$ are given in Figures \ref{fig:6m+5} and \ref{fig:6m+1}.\footnote{The families of indicated in Figures \ref{fig:6m+5} and \ref{fig:6m+1} generalise to give abelian sandpile groups isomorphic to $\bZ_{6m+5}\oplus\bZ_{6m+5}$, for all $m\geq 1$ and $\bZ_{3m+1}\oplus\bZ_{3m+1}$, for all $m\geq 1$, respectively. However, we do not require these more general results to prove Theorem \ref{thm:rank2}.}
\end{proof}

\begin{figure}
\begin{center}
\scalebox{0.9}
{\begin{tikzpicture}[fill=gray!50, scale=1, vertex/.style={circle,inner sep=2,fill=black,draw}, dot/.style={circle,inner sep=0.7,fill=black,draw}]

\coordinate (a1) at (2,1.5);
\coordinate (am1) at (4,1.5);
\coordinate (am) at (5,1.5);
\coordinate (b) at (0,2.5);
\coordinate (c) at (2,3);
\coordinate (d) at (0,0.5);
\coordinate (e) at (2,0);
\coordinate (f) at (1,1.5);


\draw [thick,->-=.6, -<-=.4] (f) -- (c);
\draw [thick,->-=.6, -<-=.4] (b) -- (c);
\draw [thick,->-=.6, -<-=.4] (b) -- (d);
\draw [thick,->-=.6, -<-=.4] (b) -- (f);
\draw [thick,->-=.6, -<-=.4] (d) -- (f);
\draw [thick,->-=.6, -<-=.4] (d) -- (e);
\draw [thick,->-=.6, -<-=.4] (e) -- (f);
\draw [thick,->-=.6, -<-=.4] (a1) -- (f);
\draw [thick,->-=.6, -<-=.4] (am1) -- (am);

\draw [thick,->-=.6, -<-=.4] (c) to [bend left=20](am);
\draw [thick,->-=.6, -<-=.4] (e) to [bend right=20](am);

\draw [thick] (a1) -- (2.5,1.5);
\draw [thick] (am1) -- (3.5,1.5);


\node at (a1) [vertex,label=north:$\alpha_{1}$]{};
\node at (am1) [vertex,label=north:$\alpha_{m-1}$]{};
\node at (am) [vertex,label=east:$\alpha_{m}$]{};

\node at (b) [vertex]{};
\node at (c) [vertex]{};
\node at (d) [vertex]{};
\node at (e) [vertex]{};
\node at (f) [vertex]{};

\node at (2.85,1.5) [dot]{};
\node at (3,1.5) [dot]{};
\node at (3.15,1.5) [dot]{};


\node at (-0.5,1.5){$m$};

\end{tikzpicture}}
\end{center}
\caption{
Directed Eulerian spherical embedding of a digraph with abelian sandpile group isomorphic to $\bZ_{6m+5}\oplus\bZ_{6m+5}$, when $m\in\{1,3,9,11,21,31,36\}$.
}
\label{fig:6m+5}
\end{figure}
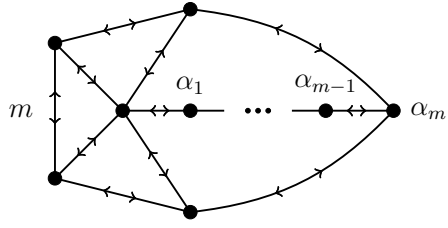

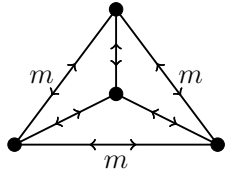
\begin{figure}
\begin{center}
\scalebox{0.9}
{\begin{tikzpicture}[fill=gray!50, scale=1, vertex/.style={circle,inner sep=2,fill=black,draw}, dot/.style={circle,inner sep=0.7,fill=black,draw}]

\coordinate (a) at (1.5,0.75);
\coordinate (b) at (3,0);
\coordinate (c) at (1.5,2);
\coordinate (d) at (0,0);


\draw [thick,->-=.6, -<-=.4] (a) -- (b);
\draw [thick,->-=.6, -<-=.4] (a) -- (c);
\draw [thick,->-=.6, -<-=.4] (a) -- (d);
\draw [thick,->-=.6, -<-=.4] (b) -- (c);
\draw [thick,->-=.6, -<-=.4] (d) -- (c);
\draw [thick,->-=.6, -<-=.4] (d) -- (b);


\node at (a) [vertex]{};
\node at (b) [vertex]{};
\node at (c) [vertex]{};
\node at (d) [vertex]{};


\node at (1.5,-0.25){$m$};
\node at (2.6,1){$m$};
\node at (0.4,1){$m$};

\end{tikzpicture}}
\end{center}
\caption{Directed Eulerian spherical embedding of a digraph with abelian sandpile group isomorphic to $\bZ_{3m+1}\oplus\bZ_{3m+1}$, when $n\in\{2,6,10,14,26,34,110,154\}$.
}
\label{fig:6m+1}
\end{figure}

\subsection{Questions}

We conclude with three questions for future consideration. The first two address the remaining cases to be considered in order to resolve Question \ref{ques:main}.

\begin{question}
Let $p\neq 2$ be a prime, $n\geq 3$ if $p>7$ and $n\geq 2$ if $p=3$ or $5$; does there exist a spherical latin bitrade with canonical group is isomorphic to $\bZ_p^n$?
\end{question}

\begin{question}
Let $p$ be a prime and let $2\leq a_1,a_2,\ldots,a_k$. If $n>1+2\sum_{i=1}^k(a_i-1)$, does there exist a spherical latin bitrade with canonical group is isomorphic to $$\bZ_p^n\oplus\left(\bigoplus_{i=1}^k\mathbb{Z}_{pa_i}\right)?$$
\end{question}

Our final question arises naturally in response to the non-existence result Theorem \ref{thm:non-existence}. For a separated, connected latin bitrade $(A,B)$ of genus greater than zero, the group $\cA_W$  is isomorphic to $\bZ
\oplus\bZ\oplus \cC$, but the minimal abelian representation (if one exists) is now a quotient of $\cC$, \cite[Theorem 6]{BlackburnTMcC}. Hence, we ask the following. 

\begin{question}
Does there exist a family of separated, connected latin bitrades for which  the minimum abelian representation of one (or both) of the partial latin squares is isomorphic to $\mathbb{Z}_2^k$ for arbitrary $k$? If so does such a family exist for a fixed genus?
\end{question}

\subsection*{Acknowledgements}
The authors express their thanks to the London Mathematical Society for a grant which enabled this research to by undertaken.

\end{document}